\UseRawInputEncoding
\documentclass[12pt]{article}
\usepackage[centertags]{amsmath}
\usepackage{amsfonts}
\usepackage{amssymb}
\usepackage{latexsym}
\usepackage{amsthm}
\usepackage{newlfont}
\usepackage{graphicx}
\usepackage{listings}
\usepackage{booktabs}
\usepackage{abstract}
\usepackage{enumerate}
\usepackage{xcolor}
\RequirePackage{srcltx}
\date{}

\newlength{\defbaselineskip}
\newcommand{\setlinespacing}[1]%
           {\setlength{\baselineskip}{#1 \defbaselineskip}}

\newcommand{\N}{{\mathbb{N}}}

\newcommand{\actaqed}{\hfill $\actabox$}
{\medskip\noindent \textit{Proof of #1. }}%
{\actaqed \medskip}

\def\D{{\mathcal D}}

\def\C{{\mathcal C}}
\def\cC{{\mathcal C}}

\def \Tr{\mathcal T}

\def \cX{\mathcal X}

\def \cM{\mathcal M}
\def\R{{\mathbb R}}
\def\Z{\mathbb Z}

\def \T{\mathbb T}

\def\bbC{\mathbb C}
\def \<{\langle}
\def\>{\rangle}

\def \Og{\Omega}

\def \ff{\varphi}
\def\al{\alpha}
\def\bt{\beta}
\def \ro{\varrho}

\def \sp{\operatorname{span}}

\def \conv{\operatorname{conv}}

\def\bx{\mathbf x}
\def\by{\mathbf y}

\def\bk{\mathbf k}

\def\bs{\mathbf s}

\def\bW{\mathbf W}

\def\bF{\mathbf F}

\def\bA{\mathbf A}

\def\bt{\beta}

\newtheorem{Theorem}{Theorem}[section]
\newtheorem{Lemma}{Lemma}[section]
\newtheorem{Definition}{Definition}[section]

\newtheorem{Corollary}{Corollary}[section]
\numberwithin{equation}{section}

\newcommand{\be}{\begin{equation}}
\newcommand{\ee}{\end{equation}}

\def\al{{\alpha}}

\def\Bl{\Bigl}
\def\Br{\Bigr}
\def\f{\frac}

\def\vi{\varphi}

\def\CD{{\mathcal D}}

\def\NN{{\mathbb N}}

\def\Og{\Omega}

\def\sub{\substack}

\def\cX{\mathcal{X}}

\def\spn{\operatorname{span}}
\def\bx{\mathbf{x}}

\def\by{\mathbf{y}}

\def\bx{\mathbf{x}}

\DeclareFontEncoding{FMS}{}{}
\DeclareFontSubstitution{FMS}{futm}{m}{n}
\DeclareFontEncoding{FMX}{}{}
\DeclareFontSubstitution{FMX}{futm}{m}{n}
\DeclareSymbolFont{fouriersymbols}{FMS}{futm}{m}{n}
\DeclareSymbolFont{fourierlargesymbols}{FMX}{futm}{m}{n}
\DeclareMathDelimiter{\VT}{\mathord}{fouriersymbols}{152}{fourierlargesymbols}{147}

\begin{document}

\title{Sampling recovery on function classes with a structural condition}

\author{ A.P. Solodov and   V.N. Temlyakov 	\footnote{
		This work was supported by the Russian Science Foundation under grant no. 23-71-30001, https://rscf.ru/project/23-71-30001/, and performed at Lomonosov Moscow State University.
  }}

\newcommand{\Addresses}{{
  \bigskip
  \footnotesize

 A.P. Solodov, \textsc{   Lomonosov Moscow State University; \\ Moscow Center of Fundamental and Applied Mathematics.}
 

 \medskip 
  V.N. Temlyakov, \textsc{ Steklov Mathematical Institute of Russian Academy of Sciences, Moscow, Russia;\\ Lomonosov Moscow State University; \\ Moscow Center of Fundamental and Applied Mathematics; \\ University of South Carolina, USA.
  \\
E-mail:} \texttt{temlyakovv@gmail.com}

}}
\maketitle

\begin{abstract}{Sampling recovery on some function classes is studied in this paper. 
Typically, function classes are defined by imposing  smoothness conditions. It was understood in nonlinear approximation that structural conditions in the form of control 
of the number of big coefficients of an expansion of a function with respect to a given system of functions
plays an important role. Sampling recovery on smoothness classes is an area of active 
research, some problems, especially in the case of mixed smoothness classes, are still open. It was discovered recently that universal sampling discretization and nonlinear sparse 
approximations are useful in the sampling recovery problem. This motivated us to systematically study sampling recovery on function classes with a structural condition. 
Some results in this direction are already known. In particular, the classes defined by 
conditions on coefficients with indices from the domains, which are differences of two dyadic cubes, are studied in the recent paper of the second author. In this paper we concentrate on studying 
function classes defined by 
conditions on coefficients with indices from the domains, which are differences of two dyadic 
hyperbolic crosses. 
		}
\end{abstract}

{\it Keywords and phrases}: Sampling discretization, universality, recovery, hyperbolic cross.

{\it MSC classification 2000:} Primary 65J05; Secondary 42A05, 65D30, 41A63.

\section{Introduction}
\label{I}

This paper is a followup of the recent papers 
\cite{KoTe}, \cite{VT202}, and \cite{VT203}.  The reader can find a discussion of the previous results directly 
connected with this area of research in \cite{KoTe}, \cite{VT202}, and \cite{VT203}. Results on discretization can be found in the recent survey papers \cite{DPTT}, \cite{KKLT}, and \cite{LMT}. 

The main goal of this paper is to study optimal sampling recovery on specific classes of multivariate functions. 	The reader can find a discussion of this type of classes in \cite{VT203}, Section 6. We give a brief description of previous results, which are directly related to results of this paper, and some of main results of this paper. All necessary notations and definitions are given in Section \ref{DN}.

We use $C$, $C'$ and $c$, $c'$ to denote various positive constants. Their arguments indicate the parameters, which they may depend on. Normally, these constants do not depend on a function $f$ and running parameters $m$, $v$, $u$. We use the following symbols for brevity. For two nonnegative sequences $a=\{a_n\}_{n=1}^\infty$ and $b=\{b_n\}_{n=1}^\infty$ the relation $a_n\ll b_n$ means that there is  a number $C(a,b)$ such that for all $n$ we have $a_n\le C(a,b)b_n$. Relation $a_n\gg b_n$ means that 
 $b_n\ll a_n$ and $a_n\asymp b_n$ means that $a_n\ll b_n$ and $a_n \gg b_n$. 
 For a real number $x$ denote $[x]$ the integer part of $x$, $\lceil x\rceil$ -- the smallest integer, which is 
 greater than or equal to $x$.

 For a function class $\bF\subset \cC(\Omega)$,  we  define (see \cite{TWW})
$$
\varrho_m^o(\bF,L_p) := \inf_{\xi } \inf_{\cM} \sup_{f\in \bF}\|f-\cM(f(\xi^1),\dots,f(\xi^m))\|_p,
$$
where $\cM$ ranges over all  mappings $\cM : \bbC^m \to   L_p(\Omega,\mu)$  and
$\xi$ ranges over all subsets $\{\xi^1,\cdots,\xi^m\}$ of $m$ points in $\Og$. 
Here, we use the index {\it o} to mean optimality.  

In \cite{VT202}   the behavior of $\varrho_m^o(\bF,L_p)$ for a new kind of classes -- $\bA^r_\bt(\Psi)$ was studied.    The order of $\varrho_m^o(\bA^r_\bt(\Psi),L_p)$ (up to a logarithmic in $m$ factor) in the case $1\le p\le 2$ for the class of uniformly bounded orthonormal systems $\Psi$ was obtained.
We give a definition of classes $\bA^r_\bt(\Psi)$.  For a given $1\le p\le \infty$,
let $\Psi =\{\psi_{\bk}\}_{\bk \in \Z^d}$, $\psi_\bk \in \cC(\Omega)$, $\|\psi_\bk\|_p \le B$, $\bk\in\Z^d$,  be a system in the space $L_p(\Omega,\mu)$. We consider functions representable in the form of an  absolutely convergent series
\be\label{Irepr}
f = \sum_{\bk\in\Z^d} a_\bk(f)\psi_\bk,\quad \sum_{\bk\in\Z^d} |a_\bk(f)|<\infty.
\ee
For $\bt \in (0,1]$ and $r>0$ consider the following class $\bA^r_\bt(\Psi)$ of functions $f$, which have representations (\ref{Irepr}) satisfying the following conditions
\be\label{IAr}
  \left(\sum_{[2^{j-1}]\le \|\bk\|_\infty <2^j} |a_\bk(f)|^\bt\right)^{1/\bt} \le 2^{-rj},\quad j=0,1,\dots  .
\ee

In the special case, when $\Psi$ is the trigonometric system $\Tr^d := \{e^{i(\bk,\bx)}\}_{\bk\in \Z^d}$   the following lower bound  
\be\label{Arlb}
\varrho_m^o(\bA^r_\bt(\Tr^d),L_1) \gg m^{1/2-1/\bt-r/d}
\ee 
was proved in \cite{VT202} (see Theorem 4.6 there).
 In \cite{VT202}   this lower bound was complemented by the following upper bound. 
 Assume that $\Psi$ is a uniformly bounded $\|\psi_\bk\|_\infty \le B$, $\bk\in\Z^d$, orthonormal system.
Let $1\le p\le 2$, $\bt \in (0,1]$, and $r>0$, then
$$
 \varrho_{m}^{o}(\bA^r_\bt(\Psi),L_p(\Omega,\mu))\le  \varrho_{m}^{o}(\bA^r_\bt(\Psi),L_2(\Omega,\mu))
$$
\be\label{Arub}
  \ll  (m(\log(m+1))^{-5})^{1/2 -1/\bt-r/d} .  
\ee
It is noted in \cite{VT202} that $(\log(m+1))^{-5}$ in (\ref{Arub}) can be replaced by $(\log(m+1))^{-4}$. 
Bounds (\ref{Arlb}) and (\ref{Arub}) show that generally for uniformly bounded orthonormal systems $\Psi$, for instance for the trigonometric system $\Tr^d$, the gap between the upper and the lower bounds is 
in terms of an extra logarithmic in $m$ factor.  

In the paper \cite{VT203}   the case $2\le p <\infty$ was studied. For a uniformly bounded Riesz system $\Psi$ Corollary 5.1 of \cite{VT203} gives for $2\le p<\infty$
\begin{equation}\label{I12}
 \varrho_{m}^{o}(\bA^r_\bt(\Psi),L_p(\Omega,\mu)) \ll  \left(\frac{m}{(\log m)^4}\right)^{1-1/p -1/\bt -r/d} .  
\end{equation}
Moreover, the bound (\ref{I12}) is provided by the simple greedy algorithm -- Weak Orthogonal Matching Pursuit (WOMP) (see the definition below). In the special case of the $d$-variate trigonometric system $\Tr^d$ the following bounds were obtained in \cite{VT203} (see (1.13) there)
\begin{equation}\label{I13}
 m^{1-1/p -1/\bt -r/d}\ll \varrho_{m}^{o}(\bA^r_\bt(\Tr^d),L_p) \ll  \left(\frac{m}{(\log m)^3}\right)^{1-1/p -1/\bt -r/d} .  
\end{equation}

We now proceed to other kind of multivariate function classes, which we study in this paper. 
Let  $\Psi=\{\psi_\bk\}_{\bk\in \Z^d}$ be as above  in the definition of classes $\bA^r_\bt(\Psi)$. 
We now define classes $\bW^{a,b}_{A_\bt}(\Psi)$. 

Let $\mathbf s=(s_1,\dots,s_d )$ be a  vector  whose  coordinates  are
nonnegative integers and
$$
\rho(\mathbf s) := \bigl\{ \mathbf k\in\mathbb Z^d:[ 2^{s_j-1}] \le
|k_j| < 2^{s_j},\qquad j=1,\dots,d \bigr\}.
$$
For 
\be\label{R9a}
f=\sum_{\bk}a_\bk(f) \psi_\bk,\qquad \sum_{\bk}|a_\bk(f) | <\infty,
\ee
 denote
$$
\delta_\bs(f,\Psi):= \sum_{\bk\in \rho(\bs)}a_\bk(f) \psi_\bk,\quad f_j:=\sum_{\|\bs\|_1=j}\delta_\bs(f,\Psi), \quad j\in \N_0,\quad \N_0:=\N\cup \{0\}
$$
and   for $\bt\in (0,1]$
$$
|f|_{A_\bt(\Psi)} := \left(\sum_{\bk}|a_\bk(f)|^\bt\right)^{1/\bt}.
$$
For parameters $ a\in \R_+$, $ b\in \R$ define the class $\bW^{a,b}_{A_\bt}(\Psi)$ of functions $f$ for which 
there exists a representation (\ref{R9a}) satisfying
\be\label{R10}
 |f_j|_{A_\bt(\Psi)} \le 2^{-aj}(\bar j)^{(d-1)b},\quad \bar j:=\max(j,1), \quad j\in \N_0.
\ee

In the case, when $\Psi$ is the trigonometric system and $\bt=1$, classes $\bW^{a,b}_{A_\bt}(\Psi)$
were introduced in \cite{VT150}. The general definition in the case $\bt=1$ is given in \cite{DTM2}. We use the notation $\bW^{a,b}_{A}(\Psi):= \bW^{a,b}_{A_1}(\Psi)$.

The following Theorem \ref{HT3} was proved in \cite{KoTe} (see Theorem 5.3 and Proposition 5.1 there).

\begin{Theorem}[{\cite{KoTe}}]\label{HT3}  Let $p\in [2,\infty)$. There exist two constants $c(a,p,d)$ and $C(a,b,p,d)$
such that we have the bound
\begin{equation}\label{sr3}
 \varrho_{m}^{o}(\bW^{a,b}_A,L_p(\T^d)) \le C(a,b,p,d)  v^{-a-1/p} (\log v)^{(d-1)(a+b)}
\end{equation}
    for any $v\in\N$ and any $m$ satisfying
$$
m\ge c(a,d,p) v(\log(2v))^4.
$$
\end{Theorem} 

We now formulate some of the main results of this paper, obtained in Section \ref{N}. 
 
\begin{Theorem}\label{MT1}  Assume that $\Psi$ is a uniformly bounded Riesz system (in the space  $L_2(\Omega,\mu)$) satisfying (\ref{ub}) and \eqref{Riesz} for some constants $0<R_1\leq R_2<\infty$.
Let $2\le p<\infty$  and $a>0$.  
 There exist constants $c=c(a,b,\bt,p,R_1,R_2,d)$ and $C=C(a,b,\bt,p,d)$ such that   we have the bound   
 \begin{equation}\label{R4w}
 \varrho_{m}^{o}(\bW^{a,b}_{A_\bt}(\Psi),L_p(\Omega,\mu)) \le C  v^{1-1/p -1/\bt -a}  (\log(2v))^{(d-1)(a+b)} 
\end{equation}
for any $m$ satisfying 
$$
m\ge c   v  (\log(2v ))^4 .
$$
Moreover, this bound is provided by the WOMP.
\end{Theorem}
We now formulate a version of Theorem \ref{MT1} in the special case of the trigonometric system $\Tr^d$. 
\begin{Theorem}\label{MT2}  Assume that $\Psi$ is the trigonometric system $\Tr^d$. Let $2\le p<\infty$  and $a>0$.  
 There exist constants $c=c(a,b,\bt,p,d)$ and $C=C(a,b,\bt,p,d)$ such that   we have the bound   
 \begin{equation}\label{Mtr}
 \varrho_{m}^{o}(\bW^{a,b}_{A_\bt}(\Tr^d),L_p(\T^d)) \le C  v^{1-1/p -1/\bt -a}  (\log(2v))^{(d-1)(a+b)} 
\end{equation}
for any $m$ satisfying 
$$
m\ge c   v  (\log(2v ))^3 .
$$
Moreover, this bound is provided by the WOMP.
\end{Theorem}

\section{Historical remarks}
\label{H}

For the reader's convenience we give some brief comments on the history of sampling recovery of the multivariate functions with mixed smoothness. The classes $\bW^{a,b}_{A_\bt}(\Psi)$, which we study in this paper, are close in a certain sense to the classes known in approximation theory under different names: Classes of functions with mixed smoothness, classes of functions with dominating mixed derivative, and others (see, for instance, \cite{VTmon}, \cite{VTbookNS}, \cite{DTU}, and \cite{VTbookMA}). An important difference between classes of functions with mixed smoothness and classes $\bW^{a,b}_{A_\bt}(\Psi)$ is that the definition of our new classes is not based on
smoothness properties of functions and instead is based on the structural properties. The structural properties are formulated in terms of coefficients of expansions with respect to the system $\Psi$. At this stage we use the hyperbolic crosses, which play an important role in the study of classes of functions with mixed smoothness. This makes the classes $\bW^{a,b}_{A_\bt}(\Psi)$ close to the classes of functions with mixed smoothness. For illustrational purposes we give a brief historical comments on the sampling recovery 
results for classes of functions with dominating mixed derivative. We define these classes momentarily.  
We begin with the case of univariate periodic functions. Let for $r>0$ 
\be\label{H1}
F_r(x):= 1+2\sum_{k=1}^\infty k^{-r}\cos (kx-r\pi/2) 
\ee
and
$$
W^r_q := \{f:f=\varphi \ast F_r,\quad \|\varphi\|_q \le 1\},
$$
where
$$
 \quad (\varphi \ast F_r)(x):= (2\pi)^{-1}\int_\T \ff(y)F_r(x-y)dy.  
$$
It is well known that for $r>1/q$ the class $W^r_q$ is embedded into the space of continuous functions $\cC(\T)$.  

In the multivariate case for $\bx=(x_1,\dots,x_d)$ denote
$$
F_r(\bx) := \prod_{j=1}^d F_r(x_j)
$$
and
$$
\bW^r_q := \{f:f=\varphi\ast F_r,\quad \|\varphi\|_q \le 1\},
$$
where
$$ (\varphi \ast F_r)(\bx):= (2\pi)^{-d}\int_{\T^d} \ff(\by)F_r(\bx-\by)d\by.
$$
 
The above defined classes $\bW^r_q$ are classical classes of functions with {\it dominating mixed derivative} (Sobolev-type classes of functions with mixed smoothness). The reader can find results on approximation properties of these classes in the books \cite{VTbookMA} and \cite{DTU}.

Recall the setting 
 of the optimal linear recovery. For a fixed $m$ and a set of points  $\xi:=\{\xi^j\}_{j=1}^m\subset \Omega$, let $\Phi $ be a linear operator from $\bbC^m$ into $L_p(\Omega,\mu)$.
Denote for a class $\bF$ (usually, centrally symmetric and compact subset of $L_p(\Omega,\mu)$) (see \cite{VT51})
$$
\varrho_m(\bF,L_p) := \inf_{\text{linear}\, \Phi; \,\xi} \sup_{f\in \bF} \|f-\Phi(f(\xi^1),\dots,f(\xi^m))\|_p.
$$

 In the above described linear recovery procedure the approximant  comes from a linear subspace of dimension at most $m$. 
It is natural to compare $\varrho_m(\bF,L_p)$  with the 
Kolmogorov widths. Let $X$ be a Banach space and $\bF\subset X$ be a  compact subset of $X$. The quantities  
$$
d_n (\bF, X) :=  \inf_{\{u_i\}_{i=1}^n\subset X}
\sup_{f\in \bF}
\inf_{c_i} \left \| f - \sum_{i=1}^{n}
c_i u_i \right\|_X, \quad n = 1, 2, \dots,
$$
are called the {\it Kolmogorov widths} of $\bF$ in $X$. It is clear that in the case $X=H$ is a Hilbert space the best approximant from a given subspace is provided by the corresponding orthogonal projection. 

We have the following obvious inequality
$$
d_m (\bF, L_p)\le  \varrho_m(\bF,L_p).
$$

Recently an outstanding progress has been made in the sampling recovery in the $L_2$ norm. We give a very brief comment on those interesting results. For special sets $\bF$ (in the reproducing kernel Hilbert space setting) the following inequality was proved (see \cite{DKU}, \cite{NSU}, \cite{KU}, and \cite{KU2}):
\be\label{H5}
\ro_{cn}(\bF,L_2) \le \left(\frac{1}{n}\sum_{k\ge n} d_k (\bF, L_2)^2\right)^{1/2}
\ee
with an absolute constant $c>0$.  

Known results on the $d_k (\bW^r_p, L_2)$ (see, for instance, \cite{VTbookMA}, p.216): For $1<q\le 2$, $r > 1/q-1/2$ and $2<q<\infty$, $r> 0$
\be\label{H6a}
d_k(\bW^r_q,L_2) \asymp \left(\frac{(\log k)^{d-1})}{k}\right)^{r-(1/q-1/2)_+},\quad (a)_+ := \max(a,0),
\ee
 combined with the (\ref{H5}) give for $1<q\le \infty$, $r > \max(1/q,1/2)$ the following bounds
$$
 \ro_m(\bW^r_q,L_2) \ll  \left(\frac{(\log m)^{d-1})}{m}\right)^{r-(1/q-1/2)_+}, 
$$
which solve the problem of the right orders of decay of the sequences $\ro_m(\bW^r_q,L_2)$ in the case $1<q<\infty$ and $r > \max(1/q,1/2)$ because, obviously, $ \ro_m(\bW^r_q,L_2) \ge  d_m(\bW^r_q,L_2)$.

The following result was obtained in \cite{VT183}. 
\begin{Theorem}\label{BT1} Let $\bF$ be a compact subset of $\C(\Omega)$. There exist two positive absolute constants $b$ and $B$ such that
$$
\ro_{bn}(\bF,L_2) \le Bd_n(\bF,L_\infty).
$$
\end{Theorem}

This theorem, combined with the upper bounds for the Kolmogorov widths obtained in \cite{TU1},
gives the following bounds for the sampling recovery (see \cite{TU1}): Let $2<q\le\infty$ and $1/q<r<1/2$ then
\be\label{H3}
 \ro_m(\bW^r_q,L_2) \ll m^{-r} (\log m)^{(d-2)(1-r)+1}
\ee
In the case $r=1/2$ we obtain
\be\label{H4}
\ro_m(\bW^{1/2}_q,L_2) \ll m^{-1/2} (\log m)^{d/2}(\log \log m)^{3/2}.
\ee

We refer the reader for results on the sampling recovery in the uniform norm to the   papers \cite{PU} and \cite{KKT}.  
The reader can find a   discussion of these results in \cite{KKLT}, Section 2.5. 

 Let us make a brief comment on the nonlinear recovery characteristic  
$\varrho_m^o$ (see Section \ref{K} for further discussion). The authors of \cite{JUV} (see Corollary 4.16 in v3) proved the following interesting bound for $1<q<2$, $r>1/q$ and  $m\ge c(r,d,q) v(\log(2v))^3$,
\begin{equation}\label{H6}
\varrho_{m}^{o}(\bW^{r}_q,L_2(\T^d)) \le C(r,d,q)  v^{-r+1/q-1/2} (\log v)^{(d-1)(r+1-2/q)+1/2}.     
\end{equation}
In \cite{DTM2}   we proved the following bound
\be\label{H7}
\varrho_{m}^{o}(\bW^{r}_q,L_2(\T^d)) \le C'(r,d,q)  v^{-r+1/q-1/2} (\log v)^{(d-1)(r+1-2/q)}      
\end{equation}
provided that
\be\label{H8}
m\ge c'(r,d,q) v(\log(2v))^3.
\ee
Let us make some comments on bounds (\ref{H6}) and (\ref{H7}). First of all, these bounds were proved by different methods. Bound (\ref{H6}) 
was proved in \cite{JUV} with the help of a powerful result from compressed sensing theory 
and bound (\ref{H7}) is based on a result on universal discretization and uses the greedy type algorithm the WOMP.  The reader can find some results on compressed sensing in \cite{VTbook} and \cite{KaTe}.
Second, the WOMP provides a sparse with respect to the trigonometric system approximant while (\ref{H6}) does not provide a sparse one. Third, the algorithm used in \cite{JUV} uses some information of the function class ($\bW^{r}_q$ in our case). This means that for each class $\bW^{r}_q$ we need an algorithm adjusted to this class. 
The algorithm WOMP does not use any information of the function class and provides the bound (\ref{H7}).
This makes the WOMP universal for the collection of classes $\bW^{r}_q$. 
Finally, bound (\ref{H7}) is slightly sharper than (\ref{H6}). 

In the above mentioned results we study sampling recovery in the $L_2$ norm. The technique, which was used in the proofs of the bounds (\ref{H6}) and (\ref{H7}) is heavily based on the fact that we approximate in the $L_2$ norm. The following result was proved in \cite{KoTe} (see Theorem 5.5 and Proposition 5.1) very recently. 

 \begin{Theorem}[{\cite{KoTe}}]\label{HT2}  Let $1<q\le 2<p<\infty$.  There exist two constants $c=c(r,d,p,q)$ and $C=C(r,d,p,q)$ such that     we have the bound for $r>1/q$
 \begin{equation}\label{H9}
  \varrho_{m}^{o}(\bW^{r}_q,L_p(\T^d)) \le C   v^{-r+1/q-1/p} (\log v)^{(d-1)(r+1-2/q)}
\end{equation}
   for any $v\in\N$ and any $m$ satisfying
$$
m\ge c  v(\log(2v))^3.
$$
\end{Theorem}

\section{Some definitions and notations}
\label{DN}

 We begin with a brief 
description of  some  necessary concepts on 
sparse approximation.   Let $X$ be a Banach space with norm $\|\cdot\|:=\|\cdot\|_X$, and let $\D=\{g_i\}_{i=1}^\infty $ be a given (countable)  system of elements in $X$. Given a finite subset $J\subset \NN$, we define $V_J(\D):=\spn\{g_j:\  \ j\in J\}$. 
For a positive  integer $ v$, we denote by $\mathcal{X}_v(\D)$ the collection of all linear spaces $V_J(\D)$  with   $|J|=v$, and 
denote by $\Sigma_v(\D)$  the set of all $v$-term approximants with respect to $\D$; that is, 
$
\Sigma_v(\D):= \bigcup_{V\in\cX_v(\D)} V.
$
Given $f\in X$,  we define
$$
\sigma_v(f,\D)_X := \inf_{g\in\Sigma_v(\D)}\|f-g\|_X,\  \ v=1,2,\cdots.
$$ 
Moreover,   for a function class $\bF\subset X$, we define 
$$
 \sigma_v(\bF,\D)_X := \sup_{f\in\bF} \sigma_v(f,\D)_X,\quad  \sigma_0(\bF,\D)_X := \sup_{f\in\bF} \|f\|_X.
 $$
 
 In this paper we consider the case $X=L_p(\Omega,\mu)$, $1\le p<\infty$. More precisely, let $\Omega$ be a compact subset of $\R^d$ with the probability measure $\mu$ on it. By the $L_p$ norm, $1\le p< \infty$, of the complex valued function defined on $\Omega$,  we understand
$$
\|f\|_p:=\|f\|_{L_p(\Omega,\mu)} := \left(\int_\Omega |f|^pd\mu\right)^{1/p}.
$$
In the case $X=L_p(\Omega,\mu)$ we sometimes write for brevity $\sigma_v(\cdot,\cdot)_p$ instead of 
$\sigma_v(\cdot,\cdot)_{L_p(\Omega,\mu)}$.

We study systems, which have  the universal sampling discretization property. 
 
  \begin{Definition}\label{ID1} Let $1\le p<\infty$. We say that a set $\xi:= \{\xi^j\}_{j=1}^m \subset \Omega $ provides the {\it  $L_p$-universal sampling discretization}   for the collection $\cX:= \{X(n)\}_{n=1}^k$ of finite-dimensional  linear subspaces $X(n)$ if we have
 \be\label{ud}
\frac{1}{2}\|f\|_p^p \le \frac{1}{m} \sum_{j=1}^m |f(\xi^j)|^p\le \frac{3}{2}\|f\|_p^p\quad \text{for any}\quad f\in \bigcup_{n=1}^k X(n) .
\ee
We denote by $m(\cX,p)$ the minimal $m$ such that there exists a set $\xi$ of $m$ points, which
provides  the $L_p$-universal sampling discretization (\ref{ud}) for the collection $\cX$. 

We will use a brief form $L_p$-usd for the $L_p$-universal sampling discretization (\ref{ud}).
\end{Definition}

 Denote for a set $\xi$ of $m$ points from $\Omega$
\be\label{muxi}
\Og_m:=\xi:=\{\xi^1,\cdots, \xi^m\}\subset \Omega,\   \  \mu_m:= \frac1{m}\sum_{j=1}^m \delta_{\xi^j},\  \  \ \text{and}\  \ \mu_\xi := \f{\mu+\mu_m}2,
\ee
where $\delta_\bx$ denotes the Dirac measure supported at a point $\bx$.

 In this paper alike the paper \cite{VT202} we study a special case, when $\Psi$   satisfies certain restrictions. We formulate these restrictions in the form of conditions imposed on $\Psi$.

{\bf Condition A.} Assume that $\Psi$ is a uniformly bounded system. Namely, we assume that $\Psi:=\{\ff_j\}_{j=1}^\infty$ is a system of uniformly bounded functions on $\Og \subset \R^d$ such that
\be \label{ub}
\sup_{\bx\in\Og} |\vi_j(\bx)|\leq 1,\   \ 1\leq j<\infty.
\ee

{\bf Condition B1.} Assume that  $\Psi $ is an orthonormal system.

{\bf Condition B2.} Assume that  $ \Psi $ is a Riesz system, i.e.
for any $N\in\N$ and any $(a_1,\cdots, a_N) \in\bbC^N,$
\begin{equation}\label{Riesz}
R_1 \left( \sum_{j=1}^N |a_j|^2\right)^{1/2} \le \left\|\sum_{j=1}^N a_j\ff_j\right\|_2 \le R_2 \left( \sum_{j=1}^N |a_j|^2\right)^{1/2},
\end{equation}
where $0< R_1 \le R_2 <\infty$.  

{\bf Condition B3.} Assume that  $\Psi$ is a Bessel system, i.e. there exists a constant $K>0$ such that for any $N\in\N$ and   for any $(a_1,\cdots, a_N) \in\bbC^N,$
\begin{equation}\label{Bessel}
  \sum_{j=1}^N |a_j|^2 \le K  \left\|\sum_{j=1}^N a_j\ff_j\right\|^2_2 .
\end{equation}

Clearly, Condition B1 implies Condition B2 with $R_1=R_2=1$. Condition B2 implies Condition B3 with 
$K=R_1^{-2}$.

\section{Some known results}
\label{K}

We formulate some known results that are used in our analysis.
The  Weak Orthogonal Matching Pursuit (WOMP)  is a  greedy algorithm defined with respect to a given dictionary (system) $\D=\{g\}$  in a  Hilbert space  equipped with the inner product $ \<\cdot,\cdot\>$ and the norm  $\|\cdot\|_H$.   It is also known
  under the name Weak Orthogonal Greedy Algorithm (see, for instance, \cite{VTbook}).\\

{\bf Weak Orthogonal Matching Pursuit (WOMP).} Let $\D=\{g\}$ be a system of  nonzero elements in $H$ such that $\|g\|_H\le 1$.  
 Let $\tau:=\{t_k\}_{k=1}^\infty\subset [0, 1]$ be a given  sequence of weakness parameters. 
Given  $f_0\in H$, we define  a sequence  $\{f_k\}_{k=1}^\infty\subset H$ of residuals  for $k=1,2,\cdots$  inductively  as follows: 
\begin{enumerate}[\rm (1)]
	\item 
 $ g_k\in \D$  is any  element  satisfying
$$
|\langle f_{k-1},g_k\rangle | \ge t_k
\sup_{g\in \D} |\langle f_{k-1},g\rangle |.
$$
 
\item  Let  $H_k := \sp \{g_1,\dots,g_k\}$, and define 
$G_k(\cdot, \CD)$ to be the orthogonal projection operator from $H$   onto the space $H_k$ .

\item   Define the residual after the $k$th iteration of the algorithm by
\begin{equation*}
f_k := f_0-G_k(f_0, \D).
\end{equation*}
\end{enumerate}

In the case when $t_k=1$ for $k=1,2,\dots$,   WOMP is called the Orthogonal
Matching Pursuit (OMP). In this paper we only consider the case when  $t_k=t\in (0, 1]$ for $k=1,2,\dots$. Clearly, $g_k$ may not be unique.  However, all the  results formulated  below  are independent of the choice of the $g_k$. The reader can find the theory of greedy approximation in the book \cite{VTbook}.

For the sake of convenience in later applications, we use the notation  $\text{WOMP}\bigl(\D; t\bigr)_H$ to   denote the WOMP   defined with respect to  a given weakness parameter $t\in (0, 1]$ and   a system $\CD$ in a Hilbert  space $H$.

 {\bf UP($u,D$). ($u,D$)-unconditional property.}   We say that a system $\D=\{\vi_i\}_{i\in I}$ of  elements 
 in a Hilbert space $H=(H, \|\cdot\|)$ 
  is ($u,D$)-unconditional with  constant $U>0$ for some integers $1\leq u\leq D$ if for any 
   $f=\sum_{i\in A} c_i \vi_i\in \Sigma_u(\D)$ with  $A\subset I$    and  $J\subset I\setminus A$ such that   $|A|+|J| \le D$,  we have
\be\label{UP}
\|f\|\leq U\inf_{g\in V_J(\D)} \|f-g\|,
\ee
where $ V_J(\CD):=\spn\{\vi_i:\  \ i\in J\}$.

We gave the above definition  for a countable (or finite) system $\D$. Similar definition can be given for any system $\D$ as well.

 \begin{Theorem}[{\cite[Corollary I.1]{LT}}]\label{PT1} Let $\CD$ be a dictionary in a Hilbert space $H=(H, \|\cdot\|)$ having  the property  {\bf UP($u,D$)}  with   constant $U>0$ for some  integers $1\leq u\leq D$.   Let $f_0\in H$, and let $t\in (0, 1]$ be a given weakness parameter. 
  Then there exists a positive constant $c_\ast:=c(t,U)$  depending only on $t$ and  $U$ such that  the $\text{WOMP}\bigl(\D; t\bigr)_H$ applied to  $f_0$ gives
  $$
  \left\|f_{\left \lceil{c_\ast v}\right\rceil} \right\| \le C\sigma_v(f_0,\D)_H,\  \ v=1,2,\cdots, \min\Bl\{u,  [ D/{(1+c_\ast)]}\Br\},
  $$
   where  $C>1$ is an absolute constant, and $f_k$ denotes the residual of $f_0$ after the $k$-th iteration of the algorithm.

\end{Theorem}

  We will  consider the Hilbert space $L_2(\Omega_m,\mu_m)$ 
 instead of  $L_2(\Omega,\mu)$,  
where $\Omega_m=\{\xi^\nu\}_{\nu=1}^m$ is a set of points that provides a good universal discretization,  and $\mu_m$ is the uniform probability measure on $\Og_m$, i.e., 
$\mu_m\{ \xi^\nu\} =1/m$, $\nu=1,\dots,m$.    Let $\CD_N(\Omega_m)$ denote  the restriction 
of a system  $\CD_N$ on the set  $\Omega_m$.
Theorem \ref{PT2} below  guarantees that the simple greedy algorithm WOMP gives the corresponding Lebesgue-type inequality in the norm $L_2(\Omega_m,\mu_m)$, and hence 
 provides good sparse recovery. Theorem \ref{PT2} was derived in \cite{DTM2} from Theorem \ref{PT1}. 

\begin{Theorem}[{\cite{DTM2}}]\label{PT2}  Let  $\CD_N=\{\vi_j\}_{j=1}^N$ be  a uniformly bounded Riesz basis  in $L_2(\Og, \mu)$  satisfying  \eqref{ub} and \eqref{Riesz} for some constants $0<R_1\leq R_2<\infty$. 
Let $\Og_m=\{\xi^1,\cdots, \xi^m\}$ be a finite subset of $\Og$  that provides the $L_2$-universal discretization (\ref{ud}) for the collection 
$\cX_u(\CD_N)$ and a given integer $1\leq u\leq N$.   Given a weakness parameter $0<t\leq 1$,   there exists a constant integer  $c=c(t,R_1,R_2)\ge 1$  such that for any integer $0\leq v\leq u/(1+c)$ and any $f_0\in \cC(\Omega)$,   the  $$\text{WOMP}\Bl(\D_N(\Og_m); t\Br)_{L_2(\Omega_m,\mu_m)}$$    applied to $f_0$  gives 
\be\label{mp}
\|f_{cv}    \|_{L_2(\Omega_m,\mu_m)} \le C\sigma_v(f_0,\CD_N(\Omega_m))_{L_2(\Omega_m,\mu_m)}, 
\ee
and
\be\label{mp2}
\|f_{c v}\|_{L_2(\Omega,\mu)} \le C\sigma_v(f_0,\CD_N)_\infty,
\ee
where $C>1$ is an  absolute constant, and $f_k$ denotes the residual of $f_0$ after the $k$-th iteration of the algorithm.
 \end{Theorem}

Recall that  the modulus of smoothness of a Banach space  $X$ is defined as 
\begin{equation}\label{CG1}
\eta(w):=\eta(X,w):=\sup_{\sub{x,y\in X\\
		\|x\|= \|y\|=1}}\left(\f {\|x+wy\|+\|x-wy\|}2 -1\right),\  \ w>0,
\end{equation}
and that $X$ is called uniformly smooth  if  $\eta(w)/w\to 0$ when $w\to 0+$.
It is well known that the $L_p$ space with $1< p<\infty$ is a uniformly smooth Banach space with 
\be\label{CG2}
\eta(L_p,w)\le \begin{cases}(p-1)w^2/2, & 2\le p <\infty,\\   w^p/p,& 1\le p\le 2.
\end{cases}
\ee
 
 In addition to the formulated in Section \ref{DN} conditions on a system $\Psi$ we need one more property of the system $\Psi$.

{\bf $u$-term Nikol'skii inequality.} Let $1\le q\le p\le \infty$ and let $\Psi$ be a system from $L_p:=L_p(\Omega,\mu)$. For a natural number $u$ we say that the system $\Psi$ has 
the $u$-term Nikol'skii inequality for the pair $(q,p)$ with the constant $H$ if the following 
inequality holds 
\be\label{vNI}
\|f\|_p \le H\|f\|_q,\qquad \forall f\in \Sigma_u(\Psi).
\ee
 In such a case we write $\Psi \in NI(q,p,H,u)$. 
 
 Note, that obviously $H\ge 1$.

 \begin{Theorem}[{\cite[Theorem 3.1]{VT203}}]\label{NUT1}  Let  $\CD_N=\{\vi_j\}_{j=1}^N$ be  a uniformly bounded Riesz system  in $L_2(\Og, \mu)$  satisfying  \eqref{ub} and \eqref{Riesz} for some constants $0<R_1\leq R_2<\infty$. 
Let $\Og_m=\{\xi^1,\cdots, \xi^m\}$ be a finite subset of $\Og$  that provides the $L_2$-universal discretization for the collection 
$\cX_u(\CD_N)$,   $1\leq u\leq N$.   Assume in addition that $\D_N \in NI(2,p,H,u)$ with $p\in [2,\infty)$. Then, for a given  weakness parameter $0<t\leq 1$,   there exists a constant integer  $c=c(t,R_1,R_2)\ge 1$  such that for any integer $0\leq v\leq u/(1+c)$ and any $f_0\in \cC(\Omega)$,   the  
$$
\text{WOMP}\Bl(\D_N(\Og_m); t\Br)_{L_2(\Omega_m,\mu_m)}
$$    
applied to $f_0$  gives 
\be\label{mpn}
\|f_{cv}    \|_{L_2(\Omega_m,\mu_m)} \le C_0\sigma_v(f_0,\CD_N(\Omega_m))_{L_2(\Omega_m,\mu_m)}, 
\ee
\be\label{mp3n}
\|f_{c v}\|_{L_p(\Omega,\mu)} \le HC_1\sigma_v(f_0,\CD_N)_{L_p(\Omega,\mu_\xi)},
\ee
where $C_i$, $i=0,1$, are absolute constants,  $f_k$ denotes the residual of $f_0$ after the $k$-th iteration of the algorithm, and $\mu_\xi$ is defined in (\ref{muxi}).
 \end{Theorem}

 For brevity denote $L_p(\xi) := L_p(\Omega_m,\mu_m)$, where $\Omega_m=\{\xi^\nu\}_{\nu=1}^m$  and  $\mu_m(\xi^\nu) =1/m$, $\nu=1,\dots,m$. Let 
$B_v(f,\D_N,L_p(\xi))$ denote the best $v$-term approximation of $f$ in the $L_p(\xi)$ norm with 
respect to the system $\D_N$. Note that $B_v(f,\D_N,L_p(\xi))$ may not be unique. Obviously,
\be\label{D5}
\|f-B_v(f,\D_N,L_p(\xi))\|_{L_p(\xi)} = \sigma_v(f,\D_N)_{L_p(\xi)}.
\ee

We proved in \cite{VT203} the following theorem.

 \begin{Theorem}[{\cite[Theorem 3.3]{VT203}}]\label{NUT3} Let $2\le p<\infty$ and let $m$, $v$, $N$ be given natural numbers such that $2v\le N$.  Let $\D_N\subset \C(\Og)$ be  a system of $N$ elements such that $\D_N \in NI(2,p,H,2v)$. Assume that  there exists a set $\xi:= \{\xi^j\}_{j=1}^m \subset \Omega $, which provides the one-sided $L_2$-universal discretization 
  \be\label{D12}
 \|f\|_2 \le D\left(\frac{1}{m} \sum_{j=1}^m |f(\xi^j)|^2\right)^{1/2}, \quad \forall\, f\in \Sigma_{2v}(\D_N), 
\ee
  for the collection $\cX_{2v}(\D_N)$. Then for   any  function $ f \in \C(\Omega)$ we have
\be\label{D13}
  \|f-B_v(f,\D_N,L_2(\xi))\|_{L_p(\Omega,\mu)} \le 2^{1/p}(2DH +1) \sigma_v(f,\D_N)_{L_p(\Og, \mu_\xi)}
 \ee
 and
 \be\label{D14}
  \|f-B_v(f,\D_N,L_2(\xi))\|_{L_p(\Omega,\mu)} \le (2DH +1) \sigma_v(f,\D_N)_\infty.
 \ee
 
 \end{Theorem}

\section{New results}
\label{N}

We will use some known general results on best $v$-term approximations with respect to 
an arbitrary system in a Banach space. Usually, these results are proved in the case of real Banach spaces. It is convenient for us to consider complex Banach spaces, partially, because of our applications to the special case, when the system of interest is the trigonometric system $\Tr^d$. We now prove the complex version of the result known in the case of real Banach spaces. For a system $\D:=\{g\}$ denote
 $$
 A_1(\D) := \left\{f:\, f=\sum_{i=1}^\infty a_ig_i,\quad g_i\in \D,\quad \sum_{i=1}^\infty |a_i| \le 1 \right\}.
 $$

\begin{Lemma}\label{NLv} Let $\D:=\{g\}$ be a system of elements in a complex Banach space $X$ satisfying $\|g\|_X \le 1$, $g\in \D$. Assume that $X$ has the modulus of smoothness $\eta(X,w) \le \gamma w^q$, $1<q\le 2$. Then, there exists a constant $C(q,\gamma)$, which may only depend on $q$ and $\gamma$, such that
\be\label{Nv}
\sigma_v(A_1(\D),\D)_X \le C(q,\gamma) (v+1)^{1/q-1}.
\ee
Moreover, bound (\ref{Nv}) is provided by a constructive method based on greedy algorithms. 
\end{Lemma}
\begin{proof} In the case of real Banach spaces Lemma \ref{NLv} is known (see, for instance, \cite{VTbook}, p.342). The corresponding bound is provided by the Weak Chebyshev Greedy Algorithm. We derive the complex case from the real one. 
Denote $g^R :=$ Re$(g)$, $g^I :=$ Im$(g)$ the real and imaginary parts of $g$, and consider the systems
$$
\D^r_1 := \{g^R\, :\, g\in \D\},\quad \D^r_2 := \{g^I\, :\, g\in \D\}.
$$
Define $X^r_i$, $i=1,2$, to be the closure in the space $X$ of the span with respect to real numbers of $\D^r_i$, $i=1,2$. Then, clearly, $\eta(X^r_i,w) \le \eta(X,w)$, $i=1,2$. By the known real Banach space version of Lemma \ref{NLv} we obtain
\be\label{Nv1}
\sigma_v(A_1(\D^r_i),\D^r_i)_X \le C_1(q,\gamma) (v+1)^{1/q-1},\quad i=1,2.
\ee
Next, let $f\in X$ be such that 
$$
f=\sum_{k=1}^\infty c_kg_k,\quad g_k \in \D,\quad \sum_{k=1}^\infty |c_k| \le 1.
$$
Writing $c_k = a_k +ib_k$, we obtain
$$
f = \sum_{k=1}^\infty (a_kg^R_k -b_kg^I) +i \sum_{k=1}^\infty (a_kg^I_k +b_kg^R)=: f^R + i f^I.
$$
Clearly,
$$
|a_k|+|b_k| \le \sqrt{2}(a_k^2+b_k^2)^{1/2} =  \sqrt{2}|c_k|.
$$
Therefore, 
$$2^{-1/2}f^R\in \conv(A_1(\D^r_1),A_1(\D^r_2)),\qquad 2^{-1/2}f^I\in \conv(A_1(\D^r_1,A_1(\D^r_2)),
$$ 
and 
$$
\sigma_{4v}(A_1(\D),\D)_X \le 2\max_{i}\sigma_v(A_1(\D^r_i),\D^r_i)_X. 
$$
This and (\ref{Nv1}) imply Lemma \ref{NLv}.

\end{proof}

Let $Q$ be a finite set of points in $\mathbb Z^d$, we denote
$$
\Psi(Q) :=\left\{ g : g(\mathbf x) =\sum_{\mathbf k\in Q}a_{\mathbf k}
\psi_\bk\right\} .
$$

\begin{Theorem}\label{NT0}   Let $1<p<\infty$, $a>0$, $b\in \R$, and $v\in\N$. Assume that $\|\psi_\bk\|_{L_p(\Omega,\mu)} \le 1$, $\bk\in\Z^d$, with some probability measure $\mu$ on $\Omega$. Then there exist two constants $c=c(a,\bt,d,p)$, $C=C(a,b,\bt,d,p)$, such that for a given $v\in \N$ there is a set $Q\subset \Z^d$, $|Q| \le v^{c}$, which does not depend on measure $\mu$ with the following property.
There is   
a constructive method $A_{v,\mu}$ based on greedy algorithms, which provides a $v$-term approximant 
from $\Psi(Q)$  with 
the bound for $f\in \bW^{a,b}_{A_\bt}(\Psi)$
$$
\|f-A_{v,\mu}(f)\|_{L_p(\Omega,\mu)} \le C   v^{-a+1/p^*-1/\bt} (\log(2 v))^{(d-1)(a+b)},\quad p^*:=\min(p,2).      
$$
\end{Theorem}
\begin{proof}  We prove the theorem for $v\asymp 2^nn^{d-1}$, $n\in \N$. Let $f\in \bW^{a,b}_{A_\bt}(\Psi)$ and
$$
f=\sum_{\bk}a_\bk(f) \psi_\bk
$$
be its representation satisfying (\ref{R10}). We approximate $f_j$ in $L_p(\Omega,\mu)$ by 
a $2v_j$-term approximant from $\Psi(\Delta Q_j)$, $\Delta Q_j := \cup_{\|\bs\|_1=j} \rho(\bs)$. The sequence $\{v_j\}$ will be defined later. For any sequence $\{v_j\}_{j=n}^\infty$ of
 nonnegative integers such that $v_n+\sum_{j>n} 2v_j \le v$ we obtain
 $$
 \sigma_v(f,\Psi)_p \le \sum_{j=n+1}^\infty \sigma_{2v_j}(f_j, \Psi(\Delta Q_j))_p + \sigma_{v_n}(f_n,\Psi(Q_n))_p,\quad Q_n := \cup_{\|\bs\|_1\le n} \rho(\bs).
 $$
We now estimate $\sigma_{2v_j}(f_j, \Psi(\Delta Q_j))_p$ for $f\in \bW^{a,b}_{A_\bt}(\Psi)$. 
 We make this approximation in two steps. First, we take $v_j$ coefficients $a_\bk(f)$, $\bk\in \Delta Q_j$, with the largest absolute values. Denote $I(j)$ the corresponding set of indexes $\bk$. Then, using 
 the known results (see, for instance, \cite{DTU}, p.114, Lemma 7.4.1), we conclude from (\ref{R10}) that
 \be\label{R11}
 \sum_{\bk \in \Delta Q_j\setminus I(j)} |a_\bk| \le (v_j+1)^{1-1/\bt}2^{-aj} (\bar j)^{(d-1)b}.
 \ee
Since   $L_p(\Og, \mu)$ is a Banach space with $\rho(w,L_p) \le C(p)w^{p^*}$ and by our assumption 
 	\be\label{uba}
	\|\psi_\bk\|_{L_p(\Og, \mu)}\leq   1,\   \bk \in \Delta Q_j,
	\ee
 	using (\ref{Nv}), we deduce from \eqref{R11}
 	 that for every $j$  there exists an $h_j\in\Sigma_{v_j}(\Psi(\Delta Q_j))$ (provided by a constructive method based on greedy algorithms)
 	such that
 $$
 \|f_j-h_j\|_{L_p(\Og, \mu)}\leq C(p) (v_j+1)^{1/p^*-1/\bt}2^{-aj}(\bar j)^{(d-1)b}.
 $$
 We take   $\al$ such that $\al(1/\bt-1/p^*)=a/2$ and specify
$$
v_j := [2^{n-\al (j-n)}j^{d-1}],\quad j=n+1,\dots.
$$
It is clear that there exists $c_0(\al,d)\in\N$ such that $v_j=0$ for $j\ge c_0(\al,d)n$. Define $Q:= Q_{c_0(\al,d)n}$.
We point out that $Q$ does not depend on $\mu$. 
For $j\ge c_0(\al,d)n$
we set $h_j=0$.
In addition to $\{h_j\}_{j=n+1}^{c_0(\al,d)n}$ we include in the approximant 
$$
S_n(f,\Psi) := \sum_{\|\bs\|_1\le n}\delta_\bs(f,\Psi).
$$
Define
$$
A_v(f,\al) := S_n(f,\Psi)+\sum_{j > n} h_j = S_n(f,\Psi)+\sum_{j = n+1}^{c_0(\al,d)n} h_j .
$$
Clearly, $A_v(f,\al)\in \Psi(Q)$. We have built a $v$-term approximant of $f$ with 
$$
v\ll 2^nn^{d-1}  +\sum_{j> n} 2v_j \ll 2^nn^{d-1}.
$$
 The error   of this approximation in the $L_p(\Omega,\mu)$ is bounded from above by ($\kappa := 1/\bt-1/p^*$)
$$ 
\|f-A_v(f,\al)\|_{L_p(\Omega,\mu)} \le \sum_{j> n} \|f_j-h_j\|_{L_p(\Omega,\mu)}  
$$
$$
\ll \sum_{j> n} (v_j+1)^{-\kappa}2^{-aj}j^{(d-1)b}
\ll \sum_{j>  n}2^{-\kappa(n-\al(j-n))}j^{-(d-1)\kappa}2^{-aj}j^{(d-1)b} 
$$
$$
 \ll 2^{-n(a+\kappa)}n^{(d-1)(b-\kappa)}\ll v^{-a+1/p^*-1/\bt} (\log(2 v))^{(d-1)(a+b)}.
$$
This completes the proof of Theorem \ref{NT0}.

\end{proof}

The following Lemma \ref{NL1} is a corollary of Theorem \ref{NT0}.

\begin{Lemma}\label{NL1}  Assume that $\Psi$ is a uniformly bounded system of functions
\be\label{R8}
|\psi_\bk(\bx)| \le 1,\quad \bx \in \Omega, \quad \bk \in \Z^d.
\ee
Let $1<p<\infty$, $a>0$, $b\in \R$, and $v\in\N$. There exist two constants $c^*=c(a,\bt,d,p)$, $C=C(a,b,\bt,d,p)$, and a set $Q\subset \Z^d$, $|Q| \le v^{c^*}$,  
such that for any probability measure $\mu$ and any finite subset $\xi\subset \Omega$   there is   
a constructive method $A_{v,\xi}$ based on greedy algorithms, which provides a $v$-term approximant 
from $\Psi(Q)$  with 
the bound for $f\in \bW^{a,b}_{A_\bt}(\Psi)$
$$
\|f-A_{v,\xi}(f)\|_{L_p(\Omega,\mu_\xi)} \le C   v^{-a+1/p^*-1/\bt} (\log(2 v))^{(d-1)(a+b)},\quad p^*:=\min(p,2).      
$$
\end{Lemma}
\begin{proof} Condition (\ref{R8}) implies that for any probability measure $\pi$ on $\Omega$ we have 
$\|\psi_\bk\|_{L_p(\Omega,\pi)} \le 1$, $\bk\in\Z^d$. Thus we can apply Theorem \ref{NT0} with 
the measure $\mu_\xi$ instead of measure $\mu$. 
\end{proof}

 We now proceed to a new result on sampling recovery. In the proof of  Theorem \ref{RT3} below we need the following known result on the universal discretization.
 
  \begin{Theorem}[{\cite{DTM2}}]\label{RT2} Let $1\le p\le 2$. Assume that $ \D_N=\{\ff_j\}_{j=1}^N\subset \cC(\Og)$ is a  system  satisfying  the conditions  \eqref{ub} and   \eqref{Bessel} for some constant $K\ge 1$. Let $\xi^1,\cdots, \xi^m$ be independent 
 	random points on $\Og$  that are  identically distributed  according to  $\mu$. 
 	 Then there exist constants  $C=C(p)>1$ and $c=c(p)>0$ such that 
 	  given any   integers  $1\leq u\leq N$ and 
 	 $$
 	 m \ge  C Ku \log N\cdot (\log(2Ku ))^2\cdot (\log (2Ku )+\log\log N),
 	 $$
 	 the inequalities 
 	 \begin{equation}\label{Ex2}
 	 \frac{1}{2}\|f\|_p^p \le \frac{1}{m}\sum_{j=1}^m |f(\xi^j)|^p \le \frac{3}{2}\|f\|_p^p,\   \   \ \forall f\in  \Sigma_u(\D_N)
 	 \end{equation}
 hold with probability $\ge 1-2 \exp\Bl( -\f {cm}{Ku\log^2 (2Ku)}\Br)$.
\end{Theorem}

 \begin{Theorem}\label{RT3}  Assume that $\Psi$ is a uniformly bounded Riesz system (in the space  $L_2(\Omega,\mu)$) satisfying (\ref{ub}) and \eqref{Riesz} for some constants $0<R_1\leq R_2<\infty$.
Let $2\le p<\infty$  and $a>0$. For any $v\in\N$ denote $u := \lceil (1+c)v\rceil$, where $c$ is from Theorem \ref{NUT1}.  Assume in addition that $\Psi \in NI(2,p,H,u)$ with $p\in [2,\infty)$.
 There exist constants $c'=c'(a,b,\bt,p,R_1,R_2,d)$ and $C'=C'(a,b,\bt,p,d)$ such that   we have the bound   
 \begin{equation}\label{R4s}
 \varrho_{m}^{o}(\bW^{a,b}_{A_\bt}(\Psi),L_p(\Omega,\mu)) \le C' H v^{1/2 -1/\bt -a}  (\log(2v))^{(d-1)(a+b)}. 
\end{equation}
for any $m$ satisfying 
$$
m\ge c'    v  (\log(2v ))^4 .
$$
Moreover, this bound is provided by the WOMP. 
\end{Theorem}
\begin{proof} First, we use Lemma \ref{NL1} in the space $L_p(\Omega,\mu)$. In our case $B=1$ and $2\le p<\infty$, which implies that 
$p^*=2$. Consider the system $\Psi_Q := \{\psi_\bk\}_{\bk\in Q}$, $|Q| \le v^{c^*}$, from Lemma
\ref{NL1}. Note, that $Q$ does not depend on $\mu$. 
 By Theorem \ref{NUT1} with $\D_N = \Psi_Q$ and Theorem \ref{RT2} with $p=2$ and $K=R_1^{-2}$  there exist $m$ points  $\xi^1,\cdots, \xi^m\in  \Og$  with 
\be\label{mbound}
		m\leq C       u (\log N)^4,
\ee  
such that  for any given $f_0\in \cC(\Omega)$,  the WOMP with weakness parameter $t$ applied to $f_0$ with respect to the system  $\D_N(\Omega_m)$ in the space $L_2(\Omega_m,\mu_m)$ provides for 
any integer $1\le v \le u/(1+c)$
\be\label{R5}
	\|f_{cv} \|_{L_p(\Omega,\mu)} \le C'H \sigma_v(f_0,\CD_N)_{L_p(\Og, \mu_\xi)}.
	\ee
In order to bound the right side of (\ref{R5}) we apply Lemma \ref{NL1} in the space $L_p(\Og, \mu_\xi)$. 
For that it is sufficient to check that $\|\psi_\bk\|_{L_p(\Og, \mu_\xi)} \le 1$, $\bk\in\Z^d$. This follows from the assumption that $\Psi$ satisfies (\ref{ub}). Thus, by Lemma \ref{NL1} we obtain for $f_0 \in \bW^{a,b}_{A_\bt}(\Psi)$
\be\label{R6}
 \sigma_v(f_0,\Psi_Q)_{L_p(\Og, \mu_\xi)} \le CHv^{1/2 -1/\bt-a}(\log(2v))^{(d-1)(a+b)}.
 \ee
  Combining (\ref{R6}), (\ref{R5}), and taking into account (\ref{mbound}) and the bound
  $$
  N\le |Q| \le  v^{c^*},
  $$
   we complete the proof.

\end{proof}
 
\begin{Corollary}\label{RC1}  Assume that $\Psi$ is a uniformly bounded Riesz system (in the space  $L_2(\Omega,\mu)$) satisfying (\ref{ub}) and \eqref{Riesz} for some constants $0<R_1\leq R_2<\infty$.
Let $2\le p<\infty$  and $r>0$.  
 There exist constants $c=c(a,b,\bt,p,R_1,R_2,d)$ and $C=C(a,b,\bt,p,d)$ such that   we have the bound   
 \begin{equation}\label{R4w}
 \varrho_{m}^{o}(\bW^{a,b}_{A_\bt}(\Psi),L_p(\Omega,\mu)) \le C  v^{1-1/p -1/\bt -a}  (\log(2v))^{(d-1)(a+b)} 
\end{equation}
for any $m$ satisfying 
$$
m\ge c   v  (\log(2v ))^4 .
$$
Moreover, this bound is provided by the WOMP.
\end{Corollary}
\begin{proof} We prove that a system $\Psi$ satisfying the conditions of Corollary \ref{RC1} also satisfies the condition $\Psi \in NI(2,p,H,u)$ with $H= C(R_2) u^{1/2-1/p}$, which is required in Theorem \ref{RT3}. Then Corollary \ref{RC1} follows from Theorem \ref{RT3}.
Indeed, let
$$
f := \sum_{\bk \in Q} c_\bk \psi_\bk,\qquad |Q| \le u.
$$
Then
$$
\|f\|_\infty \le \sum_{\bk \in Q} |c_\bk| \le u^{1/2} \left(\sum_{\bk \in Q} |c_\bk|^{1/2}\right)^{1/2} \le u^{1/2} R_2 \|f\|_2.
$$
Using the following well known simple inequality $\|f\|_p \le \|f\|_2^{2/p}\|f\|_\infty^{1-2/p}$, $2\le p\le \infty$, we obtain
$$
\|f\|_p \le R_2^{1-2/p}u^{1/2-1/p} \|f\|_2.
$$
\end{proof}

We derived Theorem \ref{RT3} from Theorem \ref{NUT1}. We now formulate an analog of Theorem \ref{RT3}, which can be derived from Theorem \ref{NUT3} in the same way as Theorem \ref{RT3} has been derived from Theorem \ref{NUT1}.

\begin{Theorem}\label{RT4} Let $2\le p<\infty$ and let $m$, $v$ be given natural numbers.    Let $\Psi$ be  a uniformly bounded Bessel system satisfying (\ref{ub}) and (\ref{Bessel}) such that $\Psi \in NI(2,p,H,2v)$. 
There exist constants $c=c(a,b,\bt,p,K,d)$ and $C=C(a,b,\bt,p,d)$ such that   we have the bound   
 \begin{equation}\label{R4ss}
 \varrho_{m}^{o}(\bW^{a,b}_{A_\bt}(\Psi),L_p(\Omega,\mu)) \le CH v^{1/2 -1/\bt -a}  (\log(2v))^{(d-1)(a+b)}  
\end{equation}
for any $m$ satisfying 
$$
m\ge c   v  (\log(2v ))^4 .
$$
 \end{Theorem}

In the special case $\Psi = \Tr^d$ is the trigonometric system, Corollary \ref{RC1} gives 
 \begin{equation}\label{R7}
 \varrho_{m}^{o}(\bW^{a,b}_{A_\bt}(\Tr^d),L_p(\Omega,\mu)) \ll  v^{1-1/p -1/\bt -a}(\log(2v))^{(d-1)(a+b)}   
\end{equation}
 for $m\gg    v  (\log(2v ))^4$.

It is pointed out in \cite{DTM2} that known results on the RIP for the trigonometric system can be used 
for improving results on the universal discretization in the $L_2$ norm in the case of the trigonometric system. We explain that in more detail.   
 Let $M\in \N$ and $d\in \N$. Denote $\Pi(M) := [-M,M]^d$ to be the $d$-dimensional cube. Consider the system $\Tr^d(M) := \Tr^d(\Pi(M))$ of functions $e^{i(\bk,\bx)}$, $\bk \in \Pi(M)$, defined on $\T^d =[0,2\pi)^d$. Then $\Tr^d(M)$ is an orthonormal system in $L_2(\T^d,\mu)$ with $\mu$ being the normalized Lebesgue measure on $\T^d$. The cardinality of this system is $N(M):= |\Tr^d(M)| = (2M+1)^d$. We are interested in bounds on $m(\cX_v(\Tr^d(M)),2)$ in a special case, when $M\le v^c$ with some constant $c$, which may depend on $d$. Then Theorem \ref{RT2} with $p=2$
 gives 
 \be\label{DT1}
 m(\cX_v(\Tr^d(v^c)),2) \le C(c,d) v (\log (2v))^4.
 \ee
  It is  stated in \cite{DTM2} that the known results of \cite{HR} and \cite{Bour} allow us to improve the bound (\ref{DT1}):
 \be\label{HR1}
 m(\cX_v(\Tr^d(v^c)),2) \le C(c,d) v (\log (2v))^3.
 \ee
 This, in turn, implies the following estimate
  \begin{equation}\label{R9}
 \varrho_{m}^{o}(\bW^{a,b}_{A_\bt}(\Tr^d),L_p) \ll  v^{1-1/p -1/\bt -a} (\log(2v))^{(d-1)(a+b)},  
\end{equation}
 for $m\gg    v  (\log(2v ))^3$.

  \Addresses

\end{document}